\documentclass[11pt,a4paper,twoside]{amsart}

\usepackage[latin1]{inputenc}
\usepackage[T1]{fontenc}
\usepackage{hyperref}
\usepackage{graphicx}
\usepackage{tikz}
\usetikzlibrary{arrows}
\usepackage{amsmath,amssymb,mathabx,mathtools}

\newcommand{\ds}{\displaystyle}

\usepackage{enumerate}

\newtheorem{thm}{Theorem}

\newtheorem{cor}[thm]{Corollary}

\theoremstyle{definition}

\theoremstyle{remark}
\newtheorem{rem}[thm]{Remark}
\newtheorem{exmp}[thm]{Example}

\mathtoolsset{showonlyrefs=false}

\title[Bumping sequences and multispecies juggling]
{Bumping sequences and multispecies juggling}

\author[A. Ayyer]{Arvind Ayyer}
\address{Department of Mathematics, Indian Institute of Science, Bangalore - 560012, India}
\email{arvind@math.iisc.ernet.in}
\author[J. Bouttier]{J\'er\'emie Bouttier}
\address{Institut de Physique Th\'eorique, Universit\'e Paris-Saclay, CEA, CNRS, F-91191 Gif-sur-Yvette and D\'epartement de Math\'ematiques et Applications, \'Ecole normale sup\'erieure, 45 rue d'Ulm, F-75231 Paris Cedex 05}
\email{jeremie.bouttier@ipht.fr}
\author[S. Corteel]{Sylvie Corteel}
\address{LIAFA, CNRS et Universit\'e Paris Diderot, Case 7014, F-75205 Paris Cedex 13, France}
\email{sylvie.corteel@liafa.univ-paris-diderot.fr}
\author[S. Linusson]{Svante Linusson}
\address{Department of Mathematics, KTH-Royal Institute of Technology, SE-100 44, Stockholm, Sweden}
\email{linusson@math.kth.se}
\author[F. Nunzi]{Fran\c cois Nunzi}
\address{LIAFA, CNRS et Universit\'e Paris Diderot, Case 7014, F-75205 Paris Cedex 13, France}
\email{fnunzi@liafa.univ-paris-diderot.fr}
	
\thanks{The first author (A.A.) acknowledges support from a UGC Centre
  for Advanced Study grant and the Department of Science and Technology 
  grant DST/INT/SWD/VR/P-01/2014, and
  thanks LIAFA for hospitality during his stay there, where this work
  was initiated. J.B.\ and S.C.\ acknowledge financial support from
  the Agence Nationale de la Recherche via the grants ANR-08-JCJC-0011
  ``IComb'', ANR 12-JS02-001-01 ``Cartaplus'' and ANR-14-CE25-0014
  ``GRAAL'', and from the ``Combinatoire \`a Paris'' project funded by
  the City of Paris. S.L.\ acknowledges support from the Swedish
  Research Council, grant 621-2014-4780. F.N.\ acknowledges support
  from the Raman-Charpak Fellowship programme.}


\begin{document}

\begin{abstract}
  Building on previous work by four of us (ABCN), we consider further
  generalizations of Warrington's juggling Markov chains. We first
  introduce ``multispecies'' juggling, which consist in having balls
  of different weights: when a ball is thrown it can possibly bump
  into a lighter ball that is then sent to a higher position, where it
  can in turn bump an even lighter ball, etc.  We both study the case
  where the number of balls of each species is conserved and the case
  where the juggler sends back a ball of the species of its choice.
  In this latter case, we actually discuss three models: add-drop,
  annihilation and overwriting. The first two are generalisations of
  models presented in (ABCN) while the third one is new and its Markov
  chain has the ultra fast convergence property. We finally consider
  the case of several jugglers exchanging balls.  In all models, we
  give explicit product formulas for the stationary probability and
  closed form expressions for the normalisation factor if known.
\end{abstract}

\maketitle

\section{Introduction}
\label{sec:in}

Several Markov chains studied in nonequilibrium 
statistical physics are known to have, despite
nontrivial dynamics, an explicit and sometimes remarkably simple
stationary state. The most famous examples of these are one-dimensional models of hopping particles such as the asymmetric exclusion process \cite{Derrida1998}, where the stationary state satisfies the so-called matrix product representation \cite{BlytheEvans} and the zero-range process, where the stationary state is factorised \cite{EvansHanney}. 
The main reason for this simplicity is the underlying combinatorial structure
of these processes. 
Of the two examples mentioned above, a variant of the former known as the totally asymmetric simple exclusion process (TASEP), solved first in \cite{DJLS93}, has a rich combinatorial structure even when the system is generalised to include several types of particles. The latter system is known as the {\em multispecies TASEP}, and its stationary state has an explicit solution which comes from queueing theory \cite{FM07}.

The multispecies TASEP has the further exceptional property that the stationary state can also be calculated if the hopping probabilities of particles depend on their location, known as the {\em inhomogeneous
multispecies TASEP}. This was first done for the three-species case in \cite{AyyerLinusson} and the result for arbitrary species has been announced in \cite{LinussonMartin}. While the stationary state of the general inhomogeneous multispecies TASEP has an explicit description in principle, the actual formulas for the stationary probabilities can be considerably complicated.

In this paper, we will first study the multispecies variants of the basic juggling process introduced in \cite{Warrington} then extended to their inhomogeneous versions 
in \cite{EngstromLeskelaVarpanen,ABCN}. In contrast to the TASEP, as 
we will show in Theorem~\ref{thm:multijugstat}, the stationary probabilities and the partition function have elegant and compact expressions.
We then study the multispecies variants of two other juggling processes, which were also introduced in \cite{Warrington}, where the number of balls of each type can vary. 
In all of these cases, we prove analogous results; see Theorems~\ref{thm:adstat} and \ref{thm:anstat}. We also introduce a new model  where the number of balls of each type can vary
that we call the overwriting model. This model has the nice property that it converges to its stationary distribution in deterministic finite time.
In probabilistic language, this is equivalent to saying that the overwriting model has a deterministic strong stationary time.

The rest of the paper is organized as follows. In
Section~\ref{sec:multijug}, we discuss in some detail the first model,
the so-called Multispecies Juggling Markov Chain (MSJMC):
Section~\ref{sec:multijugbasic} provides its definition and the
expression for its stationary distribution, and
Section~\ref{sec:multijugenrich} is devoted to the enriched chain.
Other models with a fluctuating number of balls of each type (but with
a finite state space) are considered in Section \ref{sec:multiadd}: we
introduce the multispecies extension of the add-drop and the
annihilation models studied in \cite{ABCN} in the respective Sections
\ref{sec:ad} and \ref{sec:an}. In Section \ref{sec:ov}, we present the
overwriting model. Finally, in Section \ref{sec:sevjug},
we describe another possible extension of the juggling Markov chain of
\cite{Warrington}, that involves several jugglers.

\begin{rem} \label{rem:lumping}
Our proofs were mainly obtained by a classic combinatorial approach
which consists of introducing an enriched chain whose stationary
distribution is simpler, and which yields the original chain by a
projection or ``lumping'' procedure, see e.g.~\cite[Section
2.3.1]{LevinPeresWilmer}. Let us summarize this strategy. Suppose 
we have a Markov chain on the state space $S$ (which will be a finite
set in all cases considered here), with transition matrix $P$ (which
is a matrix with rows and columns indexed by $S$, such that all rows
sum to $1$), and for which we want to find the stationary
distribution, namely the (usually unique) row vector $\pi$ whose
entries sum up to $1$ and such that $\pi P = \pi$. The idea is to
introduce another ``enriched'' Markov chain on a larger state space
$\tilde{S}$ with transition matrix $\tilde{P}$, which has the two
following properties:
\begin{itemize}
\item its stationary distribution $\tilde{\pi}$ is ``easy'' to find
  (for instance we may guess and then check its general form because its
  entries are integers with nice factorisations, or monomials in some
  parameters of the chain),
\item it \emph{projects} to the original Markov chain in the sense that
  there exists an equivalence relation $\sim$ over $\tilde{S}$ such
  that $S$ can be identified with $\tilde{S}/\sim$ (i.e.\ the set of
  equivalence classes of $\sim$), and such that the \emph{lumping
    condition}
  \begin{equation}
    \label{eq:lumping}
    \sum_{y'\sim y} \tilde{P}_{x,y'} = P_{[x],[y]}
  \end{equation}
  is satisfied for all $x,y$ in $\tilde{S}$, where $[x]\in S$ denotes the
  equivalence class of $x$.
\end{itemize}
Then, it is straightforward to check that the stationary distribution
$\pi$ of the original Markov chain is given by
\begin{equation}
  \label{eq:pisum}
  \pi_{[x]} = \sum_{x' \sim x} \tilde{\pi}_{x'}.
\end{equation}
In principle, there may be a large number of terms in the
right-hand side of \eqref{eq:pisum}, making the resulting stationary
distribution $\pi$ nontrivial.

For all the juggling models, we  give arguments to show that the Markov chains are aperiodic and irreducible. This implies that their stationary distributions
are unique. For the enriched chains, we do not explicitly prove irreducibility, since the arguments are long-winded and not particularly interesting.
Finding a stationary distribution of the enriched chain and performing the lumping procedure is sufficient to obtain the unique  stationary distribution of the original chain.
\end{rem}

Most of the results of this paper have been previously announced in the conference proceeding \cite{ABLN}.

%________________________________________________

\section{Multispecies juggling}
\label{sec:multijug}

\subsection{Definition and stationary distribution}
\label{sec:multijugbasic}

The first model that we consider in this paper, and for which we give
the most details, is a ``multispecies'' generalisation of the so-called
Multivariate Juggling Markov Chain (MJMC) \cite{ABCN}. Colloquially
speaking, the juggler is now using balls of different weights, and
when a heavy ball collides with a lighter one, the lighter ball is bumped
to a higher position, where it can itself bump a lighter ball, and so on,
until a ball arrives at the topmost position. Should the reader find
this model unrealistic, she may instead think of a lazy referee
``juggling'' with a stack of papers of varying priorities to review:
every day the referee takes the paper on the top of the stack but,
after spending his time on other duties, decides to postpone it to a
later date, possibly bumping a less important paper further down the
stack, etc. Formally, our \emph{Multispecies Juggling Markov Chain
  (MSJMC)} is defined as follows.

Let $T$ be a fixed positive integer, and $n_1,\ldots,n_T$ be a sequence
of positive integers. The state space $St_{n_1,\ldots,n_T}$ of the
MSJMC is the set of words on the alphabet $\{1,\ldots,T\}$ containing,
for all $i=1,\ldots,T$, $n_i$ occurrences of the letter $i$ (the
letter $1$ represents the heaviest ball and $T$ the lightest
one). Of course those words have length $n=n_1+\cdots+n_T$, and there
are $\binom{n}{n_1,\ldots,n_T}$ different states.

To understand the transitions, it is perhaps best to start with an
example, by considering the word $132132$ (i.e.\ $T=3$,
$n_1=n_2=n_3=2$). The first letter $1$ corresponds to the ball
received by the juggler: it can be thrown either directly to the
rightmost position, i.e.\ to the top (resulting in the
word $321321$), or in the place of any lighter ball. Say we throw it
in place of the first $2$. This $2$ can in turn be thrown either to
the rightmost position (resulting in the word $311322$), or in the place
of a lighter ball on its right: here it can only ``bump'' the second
3, which in turn has no choice but to go to the rightmost position,
resulting in the word $311223$. This latter transition is represented
on Figure~\ref{fig:bump}.

\begin{figure}[ht]
	\centering
		\includegraphics{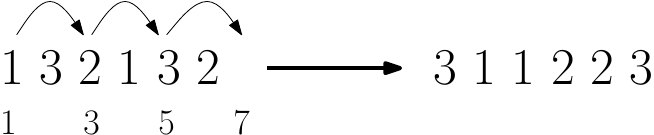}
                \caption{A possible transition from the state
                  $132132$, corresponding to the bumping sequence
                  $(1,3,5,7)$.}
	\label{fig:bump}
\end{figure}

We now give the formal definition of the transitions. Let $w=w_1\cdots
w_n$ be a state in $St_{n_1,\ldots,n_T}$, and set, by convention,
$w_{n+1}=\infty$.  A \emph{bumping sequence} for $w$ is an increasing
sequence of integers $(a(1),\ldots,a(k))$ with length at most $T+1$ such
that $a(1) = 1$, $a(k)=n+1$ and, for all $j$ between $1$ and $k-1$,
$w_{a(j)}<w_{a(j+1)}$ (that is to say, the ball at position $a(j)$ is
heavier than that at position $a(j+1)$). We denote by $\mathcal{B}_w$
the set of bumping sequences for $w$.  For $a \in \mathcal{B}_w$, we
define the state $w^a$ resulting from $w$ via the bumping sequence $a$
by
\begin{equation}
  \label{eq:wupdate}
  w^a_i =
  \begin{cases}
    w_{a(\ell-1)} & \text{if $i=a(\ell)-1$ for some $\ell$,} \\
    w_{i+1} & \text{otherwise,} \\
  \end{cases}
\end{equation}
which is easily seen to belong to $St_{n_1,\ldots,n_T}$. Returning to
the example in Figure~\ref{fig:bump} with $w=132132$, the longest possible bumping
sequence is $a=(1,3,5,7)$ and indeed $w^a=311223$.

We now define the transition probabilities, which means
assigning a probability to each bumping sequence. As in the
MJMC, these probabilities will depend on
a sequence $z_1,z_2,\ldots$ of nonnegative real parameters. 
Suppose that we have constructed
the $i-1$ first positions $(a(1),\ldots,a(i-1))$ of a random bumping
sequence, so that $a(i)$ has to be chosen in the set $\{\ell|\,
a(i-1)<\ell\leq n+1, w_{\ell}>w_{a(i-1)}\}$: $z_j$ is then
proportional to the probability that we pick $a(i)$ as the $j$'th
largest element in that set. 
 Upon normalizing, we find that the
actual probability of picking a specific $a(i)$ can be written as
\begin{equation}
  \label{eq:bumpingproba}
  Q_{w,a}(i) = \frac{z_{J_w(a(i),w_{a(i-1)})}}{y_{J_w(a(i-1),w_{a(i-1)})}},  
\end{equation}
where we introduce the useful notations
\begin{equation}
  \label{eq:yJwdef}
  \begin{split}
    y_i&=z_1+\cdots+z_i \\
    J_w(m,t)&=1+ \# \{\ell|\, m\leq\ell\leq n, w_{\ell}>t\}
  \end{split}
\end{equation}
for $m\in \{ 1,\ldots,n \}$ , $t \in \{ 1,\ldots,T \}$ and
$i \in \{ 2,\ldots,k \}$.  All in all, the global
probability assigned to the bumping sequence $a$ is $\prod_{i=2}^k
Q_{w,a}(i)$.  Noting that, for all states $w,w' \in
St_{n_1,\ldots,n_T}$, there is at most one $a \in \mathcal{B}_w$ such
that $w'=w^a$, we define the transition probability from $w$ to $w'$
as
\begin{equation}
  \label{eq:jugp}
  P_{w,w'} =
  \begin{cases}
   \ds \prod_{i=2}^k Q_{w,a}(i) &
    \text{if $w'=w^a$ for some $a\in \mathcal{B}_w$,}  \\
    0 & \text{otherwise.}
  \end{cases}
\end{equation}
For instance, the transition of Figure~\ref{fig:bump} has probability
$z_4/y_5 \times z_2/y_2 \times z_1/y_1$.

\begin{rem}
The choice of transition probabilities is very important for the model to be solvable. 
For example, if we choose $z_j$ as the probability of the $j$'th smallest instead of largest, then
 that chain does not seem to have a simple stationary distribution.
 \end{rem} 

\begin{rem}
  The MJMC \cite[Section 2]{ABCN} is recovered by taking $T=2$,
  and identifying $1$'s with balls ($\bullet$) and $2$'s with vacant
  positions ($\circ$).
\end{rem}

\begin{figure}[ht]
  \centering
  \includegraphics{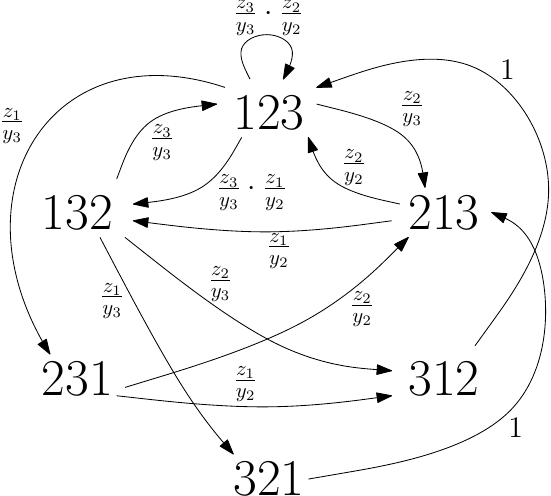}
  \caption{Transition graph of the MSJMC with $T=3$ and
    $n_1=n_2=n_3=1$.}
  \label{fig:bumpchain}
\end{figure}
\begin{exmp}
\label{mtrx}
Figure~\ref{fig:bumpchain} illustrates the MSJMC on $St_{1,1,1}$, and
the corresponding transition matrix with respect to the ordered basis
$(123,132,213,231, \allowbreak 312,321)$ reads
  \begin{equation}
    \begin{pmatrix}
      \frac{z_3}{y_3}\cdot \frac{z_2}{y_2} & \frac{z_3}{y_3}\cdot \frac{z_1}{y_2} & \frac{z_2}{y_3} & \frac{z_1}{y_3} & 0 & 0\\
      \frac{z_3}{y_3} & 0 & 0 & 0 & \frac{z_2}{y_3} & \frac{z_1}{y_3}\\
      \frac{z_2}{y_2} & \frac{z_1}{y_2} & 0 & 0 & 0 & 0\\
      0 & 0 & \frac{z_2}{y_2} & 0 & \frac{z_1}{y_2} & 0\\
      1 & 0 & 0 & 0 & 0 & 0\\
      0 & 0 & 1 & 0 & 0 & 0
    \end{pmatrix}.
  \end{equation}
  Observe that $(y_1y_2y_3,y_1^2y_3,y_1y_2^2,y_1^2y_2,y_1^2y_2,y_1^3)$ is
  a left eigenvector with eigenvalue $1$, and thus is proportional to the
  stationary distribution.
\end{exmp}

\begin{rem}
We now give an argument to show that generically (i.e.\ when all $z_i$
are nonzero), the MSJMC is irreducible and aperiodic, and thus admits a
unique stationary distribution $\pi$. 

We can reach any state $w$ by the following procedure.  Note that it
suffices to reach a cyclic shift of $w$, since we can always throw
balls to the rightmost position.  Start by positioning the heaviest
balls, that is the balls labelled $1$ as is done in $w$. This can be
done by the irreducibility of the unlabelled chain, studied in
\cite{ABCN}. Assume now by induction that balls labelled $1,\ldots,i$
have been positioned as in $w$. The juggler will now throw all balls
labelled $1,\ldots,i$ to the rightmost position and balls labelled
$i+1$ to the positions given by $w$. Those positions will be occupied
by lighter balls since all heavier balls already are sorted according
to $w$. Those lighter balls can bounce to anywhere, the rightmost
position for instance. After
$T$ transitions all balls labelled $i+1$ have been sorted according to $w$. By induction, irreducibility is proven.\\
Furthermore, the state $1^{n_1}2^{n_2}\cdots T^{n_T}$ can be sent to
itself through the bumping sequence $(1,n_1+1,n_1+n_2+1,\ldots,n+1)$,
which proves the aperiodicity of the model.
\end{rem}

Our main result for this section
is an explicit expression for $\pi$.

\begin{thm}
  \label{thm:multijugstat}
  The stationary probability of $w\in St_{n_1,\ldots,n_T}$ is given by
  \begin{equation}
    \label{eq:multijugstat}
    \pi(w) = \frac{1}{Z} \prod_{i=1}^n y_{E_w(i)},
  \end{equation}
	where
	  \begin{equation}
    \label{eq:Edef}
    E_w(i)=1+\#\{j|i\leq j\leq n, w_j>w_i\}=J_w(i,w_i),
  \end{equation}
  and the normalisation factor $Z$ reads
  \begin{equation}
    \label{eq:Zform}
    Z=\prod_{i=1}^T h_{n_i}(y_1,\ldots,y_{n-n_1-\cdots-n_i+1}),
  \end{equation}
  with $h_{\ell}$ being the complete homogeneous symmetric polynomial of
  degree $\ell$.
\end{thm}

Returning again to the example $w=132132$, we have $\pi(w) =
y_1^3y_2y_3y_5$. According to the general lumping strategy outlined in
Remark~\ref{rem:lumping}, Theorem~\ref{thm:multijugstat} is proved in
Section~\ref{sec:multijugenrich} by
introducing a suitable enriched Markov chain.

%_____________________________________________________________________

\subsection{The enriched Markov chain}
\label{sec:multijugenrich}

The first idea to define the enriched Markov chain comes from
expanding the product on the right-hand side of
\eqref{eq:multijugstat} using the definition \eqref{eq:yJwdef} of the
$y_j$'s, resulting in a sum of monomials in the $z_j$'s which is
naturally indexed by the set of sequences $v=v_1 \cdots v_n$ of
positive integers such that $v_i \leq E_w(i)$ (with $E$ as defined in 
\eqref{eq:Edef}) for all $i \in \{ 1,\ldots,n \}$. Let us call such 
$v$ an \emph{auxiliary word} for $w$. This suggests that we can define
an enriched state as a pair $(w,v)$ where $w \in St_{n_1,\ldots,n_T}$
and $v$ is an auxiliary word for $w$. We denote by 
$\mathcal{S}_{n_1,\ldots,n_T}$ the set of enriched states.

The second idea, needed to define the transitions, is to use the
auxiliary word to ``record'' some information about the past, in such a
way that all transitions leading to a given enriched state have the
same probability (this will be a key ingredient in the proof of
Theorem~\ref{thm:multijugenrichstat} below). More precisely, given an
enriched state $(w,v)$, we consider as before a bumping sequence $a
\in \mathcal{B}_w$, and we define the resulting enriched state
$(w,v)^a=(w',v')$ by updating of course the basic state as before,
i.e.\ we set $w'=w^a$ as in \eqref{eq:wupdate}, while we update the
auxiliary word as
\begin{equation}
  \label{eq:vupdate}
  v'_i=
  \begin{cases}
    E_{w'}(i)  & \text{if $i=a(\ell)-1$ for some $\ell$,} \\
    v_{i+1} & \text{otherwise.} \\
  \end{cases}
\end{equation}
For instance, in our running example $w=132132$ and $a=(1,3,5,7)$, for
$v=412211$ we have $v'=142211$.  We may think of the auxiliary word as
``labels'' carried by the balls, that are modified (maximized) for the bumped
balls and preserved otherwise. The transition probability from $(w,v)$
to $(w',v')$ is as before
\begin{equation}
  \widetilde{P}_{(w,v),(w',v')}=
  \begin{cases}
    \ds \prod_{i=2}^k Q_{w,a}(i) &
    \text{if $(w',v')=(w,v)^a$ for some $a \in \mathcal{B}_w$,}  \\
    0 & \text{otherwise.}
  \end{cases}
  \label{eq:ejugp}
\end{equation}
It is clear that the enriched chain projects to the MSJMC. Indeed, we
define an equivalence relation over $\mathcal{S}_{n_1,\ldots,n_T}$ by
simply ``forgetting'' the auxiliary word, so that the equivalence
classes may be identified with $St_{n_1,\ldots,n_T}$ (note that $1^n$
is a valid auxiliary word for any element of $St_{n_1,\ldots,n_T}$).
The lumping condition \eqref{eq:lumping} is trivially satisfied, since
we have $\widetilde{P}_{(w,v),(w,v)^a} = P_{w,w^a}$ for all $(w,v)$ in
$\mathcal{S}_{n_1,\ldots,n_T}$ and $a$ in $\mathcal{B}_w$, and
$\widetilde{P}_{(w,v),(w^a,v')}=0$ whenever $(w^a,v')\neq(w,v)^a$.

\begin{thm}
  \label{thm:multijugenrichstat}
  The stationary distribution of $(w,v)$ in $\mathcal{S}_{n_1,\ldots,n_T}$
	for the enriched chain is
  \begin{equation}
    \label{eq:multijugenrichstat}
    \tilde{\pi}(w,v)= \frac{1}{Z} \prod_{i=1}^n z_{v_i}
  \end{equation}
  where $Z$ is the normalisation factor.
\end{thm}

\begin{proof}
  We have to check that, for all $(w',v') \in
  \mathcal{S}_{n_1,\ldots,n_T}$, we have
  \begin{equation}
    \label{eq:statcond}
    \sum_{(w,v) \in \mathcal{S}_{n_1,\ldots,n_T}} \widetilde{P}_{(w,v),(w',v')}
      \tilde{\pi}(w,v) = \tilde{\pi}(w',v'),
  \end{equation}
  which is done by characterizing all possible predecessors of
  $(w',v')$. Let $(w,v)$ be such that $(w',v')=(w,v)^a$ for some
  bumping sequence $a \in \mathcal{B}_w$. We will
  show in particular that $a$ and $w$ are uniquely determined from the
  data of $(w',v')$. Hence, as claimed above, all transitions to
  $(w',v')$ have the same probability.

  We start by explaining how to recover the bumping sequence
  $a=(a(1),\allowbreak \ldots,a(k))$ or, more precisely, its set of values
  $A=\{a(1),\ldots,a(k)\}$. Recall that $1$ and $n+1$ belong to $A$ by
  definition. We claim that $j \in \{ 2,\ldots,n \}$ belongs to
  $A$ if and only if the following two conditions hold:
  \begin{itemize}
  \item[(i)] $v'_{j-1} = E_{w'}(j-1)$,
  \item[(ii)] $w'_{j-1} < w'_{j'-1}$ where $j'$ is the smallest element of
    $A \cap \{ j+1,\ldots,n+1 \}$.
  \end{itemize}
  Indeed, these two conditions are clearly necessary: (i) by
  \eqref{eq:vupdate}, and (ii) by \eqref{eq:wupdate} and the
  requirement that $w_j<w_{j'}$ when $j<j'$ are both in the bumping
  sequence. Conversely, assume that $j \notin A$, so that
  $w_j=w'_{j-1}$ and $v_j=v'_{j-1}$. By the definition of the MSJMC
  transitions, the subword $w'_j \cdots w'_n$ is a permutation of
  $w_{j+1} \cdots w_n w'_{j'-1}$. Hence, recalling \eqref{eq:yJwdef},
  $E_{w'}(j-1) - E_w(j)$ is equal to $1$ if (ii) holds
  and to $0$ otherwise. If (i) holds, we have $E_{w'}(j-1) =
  v'_{j-1} = v_j \leq E_w(j)$, and hence (ii) cannot hold. This
  completes the proof of our claim, which fully determines $A$ (hence
  $a$) by reverse induction.

  Once we have recovered $a$, it is clear that $w$ is uniquely
  determined, while we have $v_j=v'_{j-1}$ for $j \notin A$. All
  predecessors of $(w',v')$ are then obtained by picking, for each $j
  \in A \setminus \{n+1\}$, $v_j$ an arbitrary integer between $1$ and
  $E_w(j)$.  This shows that
  \begin{equation}
    \label{eq:sumpred}
    \sum_{v: (w',v')=(w,v)^a} \tilde{\pi}(w,v) =
    \frac{1}{Z} \prod_{j \notin A} z_{v'_{j-1}}
    \prod_{j \in A \setminus \{n+1\}} y_{E_w(j)}.
  \end{equation}

  The last observation we need is that
	\begin{equation}
	\label{eq:observ}
  J_w(a(i),w_{a(i-1)})=J_{w'}(a(i),w'_{a(i)-1})=E_{w'}(a(i)-1)=v'_{a(i)-1}
	\end{equation}
	for all   $i \in \{ 2,\ldots,k \}$, since $w'_{a(i)} \cdots w'_n$ is a
  permutation of $w_{a(i)} \cdots w_n$ and since
  $w_{a(i-1)}=w'_{a(i)-1}$. By \eqref{eq:bumpingproba} and
  \eqref{eq:jugp} we find that, for any predecessor $(w,v)$ of
  $(w',v')$,
  \begin{equation}
    \label{eq:probtr}
    \widetilde{P}_{(w,v),(w',v')} = \frac{\ds \prod_{j \in A\setminus \{1\}}
    z_{v'_{j-1}}}{\ds \prod_{j \in A \setminus \{n+1\}} y_{E_w(j)}}.
  \end{equation}
  Combined with \eqref{eq:sumpred}, the desired stationarity condition
  \eqref{eq:statcond} follows.
\end{proof}

\begin{proof}[Proof of Theorem~\ref{thm:multijugstat}]
  The expression~\eqref{eq:multijugstat} is immediately obtained by
  applying the general lumping expression \eqref{eq:pisum} for the
  stationary state, Theorem~\ref{thm:multijugenrichstat} and the
  definition of enriched states. It remains to check the expression
  \eqref{eq:Zform}, which we do by induction on $T$. Let
  $\phi(w)=\prod_{i=1}^n y_{E_w(i)}$, so that $Z$ is the sum of
  $\phi(w)$ over all $w \in St_{n_1,\ldots,n_T}$. The expression
  \eqref{eq:Zform} holds for $T=0$, as $Z=\phi(\epsilon)=1$ where
  $\epsilon$ is the empty word. For $T\geq 1$, let $w$ be a word in
  $St_{n_1,\ldots,n_T}$, and let $\hat{w} \in St_{n_2,\ldots,n_T}$ be
  the word obtained by removing all occurrences of $1$ in $w$, and
  shifting all remaining letters down by $1$. Denote by
  $i_1>\cdots>i_{n_1}$ the positions of $1$'s in $w$, and let
  $j_\ell=n+2-i_\ell-\ell$, so that $1 \leq j_1 \leq \cdots \leq
  j_{n_1} \leq n-n_1+1$. The mapping $w \mapsto
  (\hat{w},(j_1,\ldots,j_{n_1}))$ is bijective, and it is not
  difficult to see from the definition \eqref{eq:Edef} of
  $E$ that
  \begin{equation}
    \phi(w) = \phi(\hat{w}) \prod_{\ell=1}^{n_1} y_{j_\ell}.
  \end{equation}
  Summing the product on the right-hand side over all sequences
  $(j_1,\ldots,j_{n_1})$ yields the complete homogeneous symmetric
  polynomial $h_{n_1}(y_1,\ldots,y_{n-n_1+1})$ and \eqref{eq:Zform}
  follows by induction.
\end{proof}

%_________________________________________________________________

\section{Multispecies juggling with fluctuating types}
\label{sec:multiadd}

Our goal is here to introduce multispecies generalisations of the
add-drop and annihilation models developed in \cite{Warrington,ABCN}. 
Both models have the same state space and the same
transition graph, but different transition probabilities. The state space
$St_n^T$ is the set of words of length $n$ on the alphabet
$\mathcal{A}=\{1,\ldots,T\}$. The number of balls of each type is not
fixed anymore and thus there are $T^n$ possible states.  The
transitions are similar to the ones in the MSJMC, except that the type
of the ball the juggler throws is independent of the type of the ball
she just caught. This ball then initiates a bumping sequence as defined
before. More precisely, starting with a state $w=w_1\cdots w_n \in
St_n^T$, we let $w^- = w_2\cdots w_n$. Transitions involve
replacing the first letter of $w$ by an arbitrary $j \in \mathcal{A}$,
resulting in the intermediate state $jw^-$, then applying a bumping
sequence $a \in \mathcal{B}_{jw^-}$, resulting in the final state
$(jw^-)^a$, where $(\cdot)^a$ is defined as in
\eqref{eq:wupdate}. Defining transitions probabilities requires
specifying how we pick $j$ and $a$. The multispecies add-drop and
annihilation models differ in the way that we pick the new ball of type
$j$ and the position $a(2)$ where it is inserted, while the subsequent
elements $a(3),\ldots,a(k)$ of the bumping sequence are then chosen in
the same way as for the MSJMC. Figure~\ref{fig:fluct_chain} shows all allowed
transitions for $St_2^3$.

\begin{rem}
Both chains are irreducible, since a state $w_1\cdots w_n$ can be
reached from any state in $n$ steps by just putting a $w_i$ in the rightmost position
at the $i$'th step. The models are also aperiodic since the state $1^n$
can be sent to itself by putting a $1$ at the rightmost position. This remark also applies to the overwriting model, described later in Section~\ref{sec:ov}.
\end{rem}

\begin{figure}[ht]
	\centering
		\includegraphics{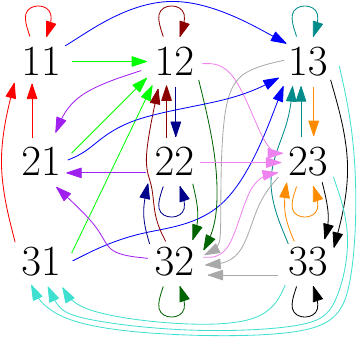}
                \caption{ The transition graphs of the add-drop 
                  and annihilation models on
                  $St_2^3$. Arrows with the same colour correspond to
                  transitions with the same probability (as, in both
                  the add-drop and the annihilation model, the
                  transition probability does not depend the first
                  letter of the initial state).}
	\label{fig:fluct_chain}
\end{figure}

%________________________________________________________________________

\subsection{Add-drop model}
\label{sec:ad}

In the add-drop model, choosing a ball of type $j$ and sending it to
the $\ell$'th available position from the right is done with
probability proportional to $c_j z_\ell$ where, in addition to the
previous parameters $z_1,z_2,\ldots$, we introduce new nonnegative
real parameters $c_1,\ldots,c_T$ that can be interpreted as
``activities'' for each type of ball. Because lighter balls can be
inserted in fewer possible positions, the actual probability of
choosing $j$ and $\ell$ has to be normalized, and reads $c_j
z_\ell/(\sum_{t=1}^T c_t y_{J_{w}(2,t)})$ where $w$ is the initial
state -- recall the definition~\eqref{eq:yJwdef} of the notations
$y_i$ and $J_w$ and note that $J_w(m,t)=J_{jw^-}(m,t)$ for all
$m>1$. As the position where the new ball is inserted is $a(2)>1$,
saying that it is the $\ell$'th available position from the right
means that $\ell=J_{w}(a(2),j)$. As described above, subsequent
elements $a(3),\ldots,a(k)$ of the bumping sequence $a$ are chosen in
the same way as for the MSJMC, so that the global probability of
picking a new ball of type $j$ and a bumping sequence $a \in
\mathcal{B}_{jw^-}$ is
\begin{equation}
  p_w(j,a)=\frac{c_j z_{J_{w}(a(2),j)}}{\ds \sum_{t=1}^T c_t y_{J_{w}(2,t)}}
  \; \prod_{i=3}^k Q_{w,a}(i),
  \label{eq:adprob}
\end{equation}
 where we recall the notation \eqref{eq:bumpingproba}.
The \emph{multispecies add-drop juggling Markov chain} is then the
Markov chain on the state space $St_n^T$ for which the transition
probability from $w$ to $w'$ reads
  \begin{equation}
	    \label{eq:adjugp}
    P_{w,w'}=
    \begin{cases}
      p_w(j,a) &
      \text{if $w'=(jw^-)^a$ for some $j \in \mathcal{A}$ and
        $a\in \mathcal{B}_{jw^-}$,}  \\
      0 & \text{otherwise.}
    \end{cases}
  \end{equation}
Note that we recover the add-drop juggling model \cite{ABCN} when we set $T=2$.

\begin{exmp}
The transition matrix of the multispecies add-drop Markov chain on the state space
$St_2^3$ in the ordered basis $(11,21,31,12,22,32,13,23,33)$ reads
\begin{equation}
    \begin{pmatrix}
     \frac{c_1z_1}{\lambda_1} & 0 & 0 & \frac{c_2z_1}{\lambda_1} & 0 & 0 & \frac{c_3z_1}{\lambda_1} & 0 & 0\\
     \frac{c_1z_1}{\lambda_1} & 0 & 0 & \frac{c_2z_1}{\lambda_1} & 0 & 0 & \frac{c_3z_1}{\lambda_1} & 0 & 0\\
     \frac{c_1z_1}{\lambda_1} & 0 & 0 & \frac{c_2z_1}{\lambda_1} & 0 & 0 & \frac{c_3z_1}{\lambda_1} & 0 & 0\\
		 0 & \frac{c_1z_1}{\lambda_2} & 0 & \frac{c_1z_2}{\lambda_2} & \frac{c_2z_1}{\lambda_2} & 0 & 0 & \frac{c_3z_1}{\lambda_2} & 0\\
		 0 & \frac{c_1z_1}{\lambda_2} & 0 & \frac{c_1z_2}{\lambda_2} & \frac{c_2z_1}{\lambda_2} & 0 & 0 & \frac{c_3z_1}{\lambda_2} & 0\\
		 0 & \frac{c_1z_1}{\lambda_2} & 0 & \frac{c_1z_2}{\lambda_2} & \frac{c_2z_1}{\lambda_2} & 0 & 0 & \frac{c_3z_1}{\lambda_2} & 0\\
		 0 & 0 & \frac{c_1z_1}{\lambda_3} & 0 & 0 & \frac{c_2z_1}{\lambda_3} & \frac{c_1z_2}{\lambda_3} & \frac{c_2z_2}{\lambda_3} & \frac{c_3z_1}{\lambda_3}\\
		 0 & 0 & \frac{c_1z_1}{\lambda_3} & 0 & 0 & \frac{c_2z_1}{\lambda_3} & \frac{c_1z_2}{\lambda_3} & \frac{c_2z_2}{\lambda_3} & \frac{c_3z_1}{\lambda_3}\\
		 0 & 0 & \frac{c_1z_1}{\lambda_3} & 0 & 0 & \frac{c_2z_1}{\lambda_3} & \frac{c_1z_2}{\lambda_3} & \frac{c_2z_2}{\lambda_3} & \frac{c_3z_1}{\lambda_3}
    \end{pmatrix}
\end{equation}
with the notation $\lambda_1 = (c_1+c_2+c_3)y_1$ , $\lambda_2 = c_1y_2+(c_2+c_3)y_1$ and $\lambda_3 = (c_1+c_2)y_2+c_3y_1$.
One can check that $(c_1^2y_1^2,c_1c_2y_1^2,c_1c_3y_1^2,c_1c_2y_1y_2,c_2^2y_1^2,c_2c_3y_1^2,\allowbreak
c_1c_3y_1y_2,c_2c_3y_1y_2,c_3^2y_1^2)$ is a left eigenvector with eigenvalue $1$.
\end{exmp}

\begin{thm}
\label{thm:adstat}
The stationary probability of $w=w_1\cdots w_n \in St_n^T$ for the add-drop
model is given by
\begin{equation}
\label{eq:adstat}
\pi(w)=\frac{1}{Z} \prod_{i=1}^n c_{w_i}y_{E_w(i)}
\end{equation}
where the normalisation factor $Z$ reads
  \begin{equation}
    \label{eq:Zadform}
    Z=\sum_{n_1+\cdots+n_T=n}\left( c_1^{n_1}\cdots c_T^{n_T}\prod_{i=1}^{T} h_{n_i}(y_1,\ldots,y_{n-n_1-\cdots-n_i+1})\right)
  \end{equation}
  with $h_{\ell}$ the complete homogeneous symmetric polynomial of
  degree $\ell$.
\end{thm}

\begin{proof}
We will again follow the strategy described in Remark~\ref{rem:lumping}.
We consider the enriched Markov chain whose state space is the set $\mathcal{S}_n^T$ of pairs of words $(w,v)$ with
$w\in St_n^T$ and $v=v_1\cdots v_n$ is an auxiliary word as defined in Section~\ref{sec:multijugenrich}.
Given $(w,v)\in \mathcal{S}_n^T$, $j\in \mathcal{A}$ and $a\in
\mathcal{B}_{jw^-}$, we define the resulting enriched state $(w,v)_j^a = (w',v')$
by setting $w'=(jw^-)^a$, and 
\begin{equation}
\label{eq:enradtrans}
v'_i = 
\begin{cases}
E_{w'}(i) & \text{if $i=a(l)-1$ for some $l$,}\\
v_{i+1} & \text{otherwise,}
\end{cases}
\end{equation}
and the transition probabilities are of course given by
\begin{equation}
\label{eq:enradprob}
\tilde{P}_{(w,v),(w',v')} = 
\begin{cases}
p_w(j,a) & \text{if $(w',v') = (w,v)_j^a$ for some $j\in \mathcal{A}$ and $a\in \mathcal{B}_{jw^-}$,}\\
0 & \text{otherwise,}
\end{cases}
\end{equation}
with $p$ as defined in \eqref{eq:adprob}.

We will now show that the stationary probability of $(w,v)\in \mathcal{S}_n^T$ for
the enriched add-drop model is given by
\begin{equation}
\label{eq:enradstat}
\tilde{\pi}(w,v)=\frac{1}{Z} \prod_{i=1}^{n} c_{w_i}z_{v_i},
\end{equation}
which will give us equation \eqref{eq:adstat} by lumping. 
  We thus have to check that, for all $(w',v') \in
  \mathcal{S}_n^T$, we have
  \begin{equation}
    \label{eq:adstatcond}
    \sum_{(w,v) \in \mathcal{S}_n^T} \widetilde{P}_{(w,v),(w',v')}
      \tilde{\pi}(w,v) = \tilde{\pi}(w',v').
  \end{equation}
Let $(w',v')$ be a state in $\mathcal{S}_n^T$. 
For a given $(w',v')$ we can deduce most of a possible predecessor $(w,v)$. As in the proof of Theorem \ref{thm:multijugenrichstat} we can first deduce the
bumping sequence $a$, then the type $j$ of the added ball. This means that $w_2,\ldots,w_n$ and 
$v_i,i\notin A$ are uniquely determined. Recall that for $a=(a(1), \dots, a(k))$ we defined $A=\{a(1),\dots, a(k)\}$.
Let $W=jw_2\cdots w_n$. We have
  \begin{equation}
    \label{eq:adsumpred}
    \sum_{(w,v) : (w',v')=(w,v)_j^a} \tilde{\pi}(w,v) =
    \frac{1}{Z} \sum_{i=1}^T c_{i} y_{J_W(2,i)}\prod_{\ell \notin A} z_{v'_{\ell-1}}
    \prod_{\ell \in A \setminus \{1,n+1\}} y_{E_W(\ell)}.
  \end{equation}
Furthermore, by observing that $J_w(a(i),w_{a(i-1)})=v'_{\ell-1}$ (as in equation \eqref{eq:observ}) we have, for all $(w,v)$ such that $(w,v)_j^a = (w',v')$,
\begin{equation}
\label{eq:adprobtr}
\tilde{P}_{(w,v),(w',v')} = 
\frac{c_j z_{J_W(a(2),j)}}{\ds \sum_{t=1}^T c_t y_{J_W(2,t)}} \;
\frac{\ds\prod_{\ell\in A\setminus \{1,a(2)\}}z_{v'_{\ell-1}}}{\ds\prod_{\ell\in A\setminus\{1,n+1\}}y_{E_W(\ell)}}.
\end{equation}
  Combined with \eqref{eq:adsumpred}, the desired stationarity condition
\eqref{eq:adstatcond} follows.
\end{proof}

%_______________________________________________________________________

\subsection{Annihilation model}
\label{sec:an}
In this section, we assume that hopping parameters are probabilities, i.e.\  
$z_1+\cdots+z_{n+1}=1$.
In this model, we consider that the juggler first tries to send a ball of
type $1$. She chooses $\ell\in\{ 1,\ldots,n+1\}$ with probability $z_{\ell}$, and tries to send the 
ball at the $\ell$'th available position (counted from the right as before). 
If $\ell$ is a valid position (that is, it is not larger than the number of available positions for the $1$), there is a bumping sequence whose subsequent elements are drawn in the same way as for the MSJMC.
Otherwise,
she tries instead to send a $2$ according to the same procedure, etc. In the end, if she did not manage to send
any ball of type in $\{ 1,\ldots,T-1\}$, she just sends a $T$ to the righmost position.
Note that failing to send a ball of type $t$ for an initial state $w$ is done with probability $1-y_{J_w(2,t)}$.
Globally, the probability of picking a new ball of type $j$ and a bumping sequence $a \in \mathcal{B}_{jw^-}$ reads
\begin{equation}
  q_w(j,a)= \begin{cases} 
  \ds z_{J_w(a(2),j)} \prod_{t=1}^{j-1}\left(1-y_{J_w(2,t)}\right) 
  \prod_{i=3}^k Q_{w,a}(i) & \text{if $j<T$,}\\
	\ds \prod_{t=1}^{T-1}\left(1-y_{J_w(2,t)}\right) & \text{if $j=T$.}
	\end{cases}
  \label{eq:annprob}
\end{equation}
In the latter case, we have $a=(1,n+1)$. The transition probabilities of the \emph{multispecies annihilation juggling Mar\-kov chain} are obtained by replacing $p_w(j,a)$ with $q_w(j,a)$ in \eqref{eq:adjugp}.
Note that we recover the annihilation juggling model \cite[Section 4.2]{ABCN} when we set $T=2$.

\begin{exmp}
The transition matrix of the multispecies annihilation Markov chain on the state space
$St_2^3$ in the basis $(11,21,31,12,22,32,13,23,\allowbreak 
33)$ reads
  \begin{equation}
    \begin{pmatrix}
     z_1 & 0 & 0 & z_1(z_2+z_3) & 0 & 0 & (z_2+z_3)^2 & 0 & 0\\
		 z_1 & 0 & 0 & z_1(z_2+z_3) & 0 & 0 & (z_2+z_3)^2 & 0 & 0\\
		 z_1 & 0 & 0 & z_1(z_2+z_3) & 0 & 0 & (z_2+z_3)^2 & 0 & 0\\
		 0 & z_1 & 0 & z_2 & z_1z_3 & 0 & 0 & (z_2+z_3)z_3 & 0\\
		 0 & z_1 & 0 & z_2 & z_1z_3 & 0 & 0 & (z_2+z_3)z_3 & 0\\
		 0 & z_1 & 0 & z_2 & z_1z_3 & 0 & 0 & (z_2+z_3)z_3 & 0\\
		 0 & 0 & z_1 & 0 & 0 & z_1z_3 & z_2 & z_2z_3 & z_3^2\\
		 0 & 0 & z_1 & 0 & 0 & z_1z_3 & z_2 & z_2z_3 & z_3^2\\
		 0 & 0 & z_1 & 0 & 0 & z_1z_3 & z_2 & z_2z_3 & z_3^2
    \end{pmatrix}.
  \end{equation}
The stationary distribution is given by 
$( z_1^2,z_1^2(1-y_1),z_1(1-y_1)^2,z_1(1-y_1),z_1^2z_3(1-y_1),\allowbreak z_1z_3(1-y_1)^2,y_2(1-y_1)^2,z_3y_2(1-y_1)^2,z_3^2(1-y_1)^2 )$, which is the unique left eigenvector with eigenvalue $1$.
\end{exmp}

\begin{thm} \label{thm:anstat}
The stationary probability of $w=w_1\cdots w_n \in St_n^T$ for the annihilation
model is given by
\begin{equation}
\label{eq:anstat}
\pi(w)= \left( \prod_{i=1,w_i<T}^n  y_{E_w(i)} 	\right)
\left( \prod_{\ell=2}^T \prod_{p=1}^{\# \{m|w_m \geq \ell\}}  \big(1-y_p\big) \right).
\end{equation}
Moreover, no normalisation factor is needed as
\begin{equation}
  \sum_{w \in St_n^T} \pi(w) =  (z_1+\cdots+z_{n+1})^{n(T-1)} = 1.
\end{equation}
\end{thm}

The stationary probabilities of enriched states
are no longer monomials in the $z_i$'s, which suggest that a further
enrichment is possible as already observed for the case $T=2$ in Section 4.2 of
\cite{ABCN}.

\begin{proof}

The theorem is proved by enriching the chain as before. 
We again use the state space $\mathcal{S}_n^T$ for the enriched multispecies annihilation Markov chain, as for the enriched multispecies add-drop Markov
chain in Section~\ref{sec:ad}. The transitions are also defined in the same way. The transition probabilities are now given by
\begin{equation}
\label{eq:enranprob}
\tilde{P}_{(w,v),(w',v')} = 
\begin{cases}
q_w(j,a) & \text{if $(w',v') = (w,v)_j^a$ for some $j\in \mathcal{A}$ and $a\in \mathcal{B}_{jw^-}$,}\\
0 & \text{otherwise,}
\end{cases}
\end{equation}
with $q$ as defined in \eqref{eq:annprob}.
We will now show that the stationary probability of $(w,v)\in \mathcal{S}_n^T$, $w=w_1\cdots w_n$,
$v=v_1\cdots v_n$ is given by
\begin{equation}
\label{eq:enranstat}
\tilde{\pi}(w,v)=\prod_{i=1,w_i < T}^n z_{v_i}\prod_{\ell=2}^T\prod_{p=1}^{\# \{m|w_m \geq \ell\}} (1-y_p).
\end{equation}
Once we prove this, we will obtain a proof of \eqref{eq:anstat} by lumping.
To do so, we have to check that $\tilde{\pi}$ satisfies
  \begin{equation}
    \label{eq:anstatcond}
    \sum_{(w,v) \in \mathcal{S}_n^T} \widetilde{P}_{(w,v),(w',v')}
      \tilde{\pi}(w,v) = \tilde{\pi}(w',v').
  \end{equation}

Let $(w',v')$ be a state in $\mathcal{S}_n^T$. 
We do not have a lot of choice in choosing a predecessor $(w,v)$ of $(w',v')$;
the bumping sequence $a$, the type $j$ of the added ball, $w_2,\ldots,w_n$ and 
$v_i,i\notin A$, where as before $A$ is the set of values in $a$, are uniquely determined
(this works exactly as for the proof of Theorem \ref{thm:multijugenrichstat}).
Let $W=jw_2\cdots w_n$, and consider the quantities
\begin{equation}
  \begin{split}
    C&=\prod_{i=1,w'_i < T}^n z_{v'_i},\\
    C'&=\begin{cases}
      z_{J_W(a(2),j)}\frac{\prod_{i\in A\setminus\{1,a(2)\}} z_{v'_{i-1}}}{\prod_{i\in A\setminus \{1,n+1\}}y_{E_W(i)}} & \text{if $j < T$,}\\
      1 &\text{otherwise,}
    \end{cases}\\
    C''&=\prod_{\substack{i\in A\setminus\{1,n+1\}\\ w_i< T}} y_{E_W(i)} \prod_{\substack{i\notin A\\ w_i< T}} z_{v_i},\\
    D&=\prod_{\ell=2}^T\prod_{p=1}^{\# \{m|w'_m \geq \ell\}} (1-y_p),\\
    D'&=\prod_{i=1}^{j-1} (1-y_{J_W(2,i)}),\\
    D''&=\prod_{\ell=2}^T\prod_{p=1}^{\#  \{m\geq 2 | w_m\geq \ell\}}(1-y_p), \\
  \end{split}
\end{equation}
\begin{equation*}
    K= \prod_{\ell=2}^T (1-y_{\#\{m\geq 2|w_m\geq
    	\ell\}+1}) + \sum_{w_1=1}^{T-1}y_{E_{w}(1)}\prod_{\ell=2}^{w_1}(1-y_{\#\{m\geq 2|w_m\geq
      \ell\}+1}).
\end{equation*}
First, we note that $C'C''=C$. Recall that $j=T$ implies $A=\{1,n+1\}$. Secondly, we have $D'D''=D$, since $J_W(2,i)=\#\{m\geq 2| w_m> i\}+1$.
For all $(w,v)$ such that $\tilde{P}_{(w,v),(w',v')}\neq 0$, we have
\begin{equation}
    \label{eq:anprobtr}
\tilde{P}_{(w,v),(w',v')}=C'D',
\end{equation}
where the $C'$ follows from \eqref{eq:observ} and \eqref{eq:bumpingproba} and where we recall that the probability of failing to send a ball of type $i$ for an initial state $w$ is done with probability $1-y_{J_w(2,i)}$.
Note that the transfer probability does not depend on the choice of $(w,v)$. 
Now we have
\begin{equation}
\label{eq:ansumpred}
\sum_{(w,v) : (w',v')=(w,v)_j^a} \tilde{\pi}(w,v) = C''D''K,
\end{equation}
where the $i=1$ case has been removed from $C''$ and collected in $K$. Furthermore, $K$ can be rewritten as
\begin{equation}
K = \prod_{\ell = 1}^{T-1} (1 - y_{J_w(2,\ell)} )+ 
\sum_{w_1=1}^{T-1} y_{J_w(2,w_1)}  \prod_{\ell = 1}^{w_1-1} (1 - y_{J_w(2,\ell)}),
\end{equation}
which can be easily seen to telescope to $1$.
The desired condition \eqref{eq:anstatcond} follows.
\end{proof}

\subsection{Overwriting model}
\label{sec:ov} 

The aim is here to describe a multispecies generalisation of the annihilation model studied in \cite{ABCN} in which
the \textit{ultrafast convergence} property holds, namely we want the stationary distribution to be reached in a finite number of steps, independent of the starting distribution.
Let $P$ be the transition matrix of a Markov chain (which is assumed to be irreducible and aperiodic). Saying that this Markov chain has the ultrafast convergence property is equivalent to saying that there exists an integer $m$ such that $P^m$ is the matrix whose rows are copies of the left eigenvector of $P$ for the
eigenvalue $1$, or to saying that $P$ has only one nonzero eigenvalue (which is $1$).

\subsubsection{Model description}

The state space $St_n^T$ is, as described in Section~\ref{sec:multiadd}, is the set of words in $\{1,\ldots,T\}^n$, and the transition probabilities are dictated by the indeterminates $z_1,\dots,z_{n+1}$ satisfying $z_1+z_2+\cdots +z_{n+1}=1$ as in Section~\ref{sec:an}.

In the {\em overwriting model}, each integer gets a chance to overwrite an integer larger than it. More formally, the transitions are described by the following process. Initially, the first letter is erased, and everything is moved to the left by $1$ step.
Now the juggler first tries to send a ball of
type $1$; she chooses $i\in \{1,n+1\}$ with probability $z_i$, and aims for the $i$'th available position, meaning those positions which contain an integer greater than 1 (counting from right as before).
If $i$ is greater than the number of available positions for a $1$ to land, the juggler simply fails to send a $1$, otherwise the $1$ lands to some position, destroying the ball previously occupying it if it isn't the righmost position.
She then tries to send a $2$, which can only land in a position higher than the $1$ did (any available position if the $1$ didn't land), or fail to land. She then tries to send a $3$, and so on.
If, after trying to send a ball of each type in $\{1,\ldots,T-1\}$, no ball landed to the rightmost position,
she puts a $T$ in that position. Otherwise, the new state is reached as soon as a ball is sent to the rightmost position.

For $w\in St_n^T$, we define an \textit{overwriting sequence} $B = ((b_1,t_1),\ldots,(b_{k},t_{k}))$ for $w$ as follows: 
$1<b_1<b_2<\ldots<b_{k}=n+1$, $1\leq t_1<t_2<\ldots<t_{k}\leq T$, and for all $i=1,\ldots,k$, $t_i< w_{b_i}$ with, by
convention, $w_{n+1}=+\infty$. We denote by $\mathcal{B}_w$ the set of overwriting sequences for $w$.

For $w\in St_n^T$ and $B\in \mathcal{B}_w$, the state in $St_n^T$ obtained by applying the overwriting sequence $B$ to the word $w$, denoted $w^B$, is given by
\begin{equation}   \label{eq:owstate}
    w^B_i = 
    \begin{cases}
      t_j & \text{if $i=b_j-1$ for some $j$,}  \\
      w_{i+1} & \text{otherwise.}
    \end{cases}
\end{equation}
Note that the probability for the juggler to fail to send a ball of type $\ell$ during the $j$'th part of the overwriting
is $1-y_{J_w(b_{j-1}+1,\ell)}$,
with the convention that $b_0 = 1$ and where $J$ is as in \eqref{eq:yJwdef}.
This returns $1$ if a ball has already been sent to the rightmost position. 
The probability for the juggler to succeed in sending a ball of type $t_j$ 
to position $b_j-1$ during the $j$'th part of the overwriting is
$z_{J_w(b_j,t_j)}$.

\begin{figure}
	\centering
		\includegraphics[width=0.9\textwidth]{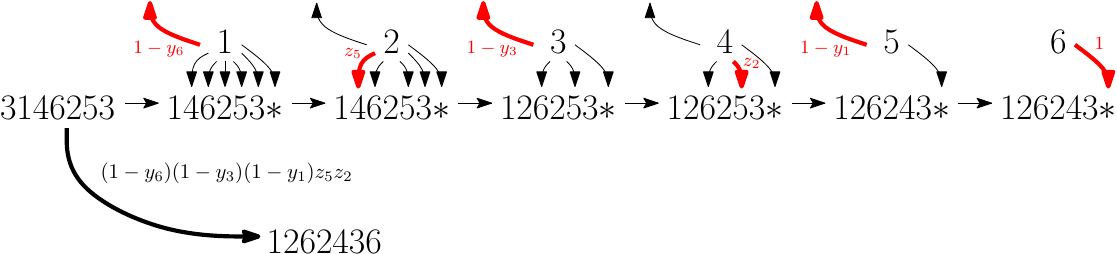}
                \caption{Computation of the transition probability
                  from $3146253$ to $1262436$ in the overwriting
                  multispecies Markov chain on $St_7^6$. We
                  successively attempt to insert each number
                  $1,\ldots,6$ at an available position or discard it
                  (the possible choices are represented by arrows, and
                  the chosen one is displayed as a thick red arrow
                  together with its probability).}
		\label{fig:owseq}
\end{figure}

\begin{exmp}
In Figure \ref{fig:owseq}, we describe the transition from the state $3146253$ to the state $1262436$ in
the state space $St_7^6$. The corresponding overwriting sequence is given by $((3,2),(6,4),(8,6))$ and the probability of this transition is $(1-y_6) z_5 (1-y_3) z_2 (1-y_1)$.
\end{exmp}

Given $z_1, \ldots, z_{n+1}$ nonnegative real numbers summing to $1$,
we are now able to define the transition probabilities for the overwriting multispecies Markov chain:
for $w,w'\in St_n^T$ if there exists $B\in\mathcal{B}_w$ such that $w'=w^B$,
the transition probability from $w$ to $w'$ reads
  \begin{equation}
	    \label{eq:anjugp}
    \mathcal{P}_{w,w'}= 
      \prod_{j=1}^{k} \Big(\prod_{\ell=t_{j-1}+1}^{t_j-1} (1-y_{J_w(b_{j-1}+1,\ell)})\Big) \prod_{j|t_j\neq T} z_{J_w(b_j,t_j)},
  \end{equation}
and it reads $0$ otherwise.

\begin{exmp}
	The transition matrix of the overwriting Markov chain on the state space
	$St_2^3$ in the ordered basis $(11,21,31,12,22,32,13,23,33)$ reads
	\begin{equation}
		\begin{pmatrix}
			z_1 & 0 & 0 & z_1(z_2+z_3) & 0 & 0 & (z_2+z_3)^2 & 0 & 0\\
			z_1 & 0 & 0 & z_1(z_2+z_3) & 0 & 0 & (z_2+z_3)^2 & 0 & 0\\
			z_1 & 0 & 0 & z_1(z_2+z_3) & 0 & 0 & (z_2+z_3)^2 & 0 & 0\\
			0 & z_1 & 0 & z_1z_2 & z_1z_3 & 0 & z_2(z_2+z_3) & z_3(z_2+z_3) & 0\\
			0 & z_1 & 0 & z_1z_2 & z_1z_3 & 0 & z_2(z_2+z_3) & z_3(z_2+z_3) & 0\\
			0 & z_1 & 0 & z_1z_2 & z_1z_3 & 0 & z_2(z_2+z_3) & z_3(z_2+z_3) & 0\\
			0 & 0 & z_1 & z_1z_2 & 0 & z_1z_3 & z_2(z_2+z_3) & z_2z_3 & z_3^2\\
			0 & 0 & z_1 & z_1z_2 & 0 & z_1z_3 & z_2(z_2+z_3) & z_2z_3 & z_3^2\\
			0 & 0 & z_1 & z_1z_2 & 0 & z_1z_3 & z_2(z_2+z_3) & z_2z_3 & z_3^2
		\end{pmatrix}.
	\end{equation}
	The stationary distribution is given by the row vector $
	(z_1^2, z_1^2 (1-y_1), \allowbreak 
	z_1 (1-y_1)^2, 
	z_1 y_2 (1-y_1),
	z_1^2 z_3 (1-y_1), z_1 z_3 (1-y_1)^2, y_2 (1-y_1)^2,
	z_3 y_2  (1-y_1)^2, z_3^2 (1-y_1)^2)$,
	which is the unique left eigenvector for the eigenvalue $1$.
\end{exmp}

The stationary distribution of the overwriting chain does not seem to have a simple formula in general, unlike the add-drop and annihilation multispecies variants. However, we do obtain an indirect formula using an enriched chain, which we state as Corollary~\ref{cor:ovstat} in Section~\ref{sec:overwr-enrich}. It turns out that the occupation probability for the last site and the joint occupation distributions at the last two sites have particularly simple expressions, which is what we state next.

\begin{thm}
\label{thm:special}
The stationary probability of having a $j$ at the last site is given by
\begin{equation}
\mathbb{P}(w_n = j) = 
\begin{cases}
z_1 (1-z_1)^{j-1} & \text{if $j<T$}, \\
(1-z_1)^{T-1} & \text{if $j=T$.}
\end{cases}
\end{equation}
The joint probability of having an $i$ at the $(n-1)$'th site and a $j$ at the $n$'th site is given by
\begin{multline}
\mathbb{P}(w_{n-1} = i,\, w_n = j) = (1-z_1)^{\max(i,j) - 1} (1-y_2)^{\min(i,j) -1}
\times \\
\begin{cases}
z_1y_2 & \text{if $i<j<T$}, \\
z_1^2 & \text{if $j\leq i<T$},\\\
y_2 & \text{if $i<j=T$},\\
z_1 & \text{if $j<i=T$},\\
1 & \text{if $i=j=T$}.
\end{cases}
\end{multline}
\end{thm}

The proof is given in the following section, where we construct an enriched chain and analysing the transitions therein.

\subsubsection{Staircase tableaux enrichment}
\label{sec:overwr-enrich}

It is now natural to look at staircase tableaux, as was done for 
the original juggling model in \cite{EngstromLeskelaVarpanen}.
The state space for the enriched version of the overwriting chain on $St_n^T$ is the set of Young tableaux of shape $(n,n-1,\ldots,2,1)$,
with the following conditions on the entries in cells.
\begin{enumerate}
	\item Entries belong to the set $\{1,\ldots,T-1\}$.
	\item Entries appear in increasing order from left to right, 
	and from bottom to top (the diagrams are drawn in French notation).
	\item Empty cells are allowed.
\end{enumerate}
This set of tableaux is denoted by $\mathcal{T}_n^T$.
For $V\in \mathcal{T}_n^T$, we will denote by $V_{*,k}$ the $k$'th column, from the left, of $V$ (which has $n+1-k$ cells). 

The transitions of the enriched Markov chain are as follows. At each step, the entries in the bottom row are deleted, all remaining entries are moved one step down and one step right, and we add entries to the leftmost column so that the tableaux conditions above still hold. More precisely, we proceed in the following way.
We first try to add a $1$ by choosing a number $k_1$ in $\{ 1,\ldots,n+1 \}$ and placing a $1$ in the $k_1$'th free position
(from top to bottom; a position is ``free'' if and only if there is no $1$ in the same row). If $k_1$ is greater than
the number of free positions, no $1$ is added.
We then similarly try to add a $2$ by choosing $k_2$ in $\{ 1,\ldots,n+1 \}$ and placing a $2$ in the $k_2$'th free position,
a position being free if there is no $1$ or $2$ in the same row, and having no $1$ above it. We continue this way until all numbers between $1$ and $T-1$ have been tried.

For $V \in \mathcal{T}_n^T$, for $i \in \{1,\ldots,T\}$ and $k \in \{1,\ldots,n\}$,
we introduce the useful notation
\begin{equation} \label{contrib}
\mathcal{C}_V(i,k) =
\begin{cases}
 z_{1+\#\{\text{cells above entry $i$ in $V_{*,k}$, with no entry $j\leq i$ to the right}\}} & \\
 \hfill \text{ if $i$ is in $V_{*,k}$,} &\\
 1-y_{\# \{\text{cells in $V_{*,k}$ with no entry $j\leq i$ in, to the right or on top}\}} & \\
 \hfill \text{ otherwise.} &
\end{cases}
\end{equation}
Here we use the convention that $y_0=0$.
\begin{exmp}
Figure \ref{fig:tabtrans} gives the example of a state in $\mathcal{T}_4^4$, and all possible states that it can transition to (with transition probabilities below the corresponding arrows). If we call $V$ the topmost tableau, we have for example $C_V(3,1)=1-y_0=1$ and $C_V(2,2)=z_2$.
\end{exmp}

\begin{figure}[ht]
	\centering
		\includegraphics[width=0.7\textwidth]{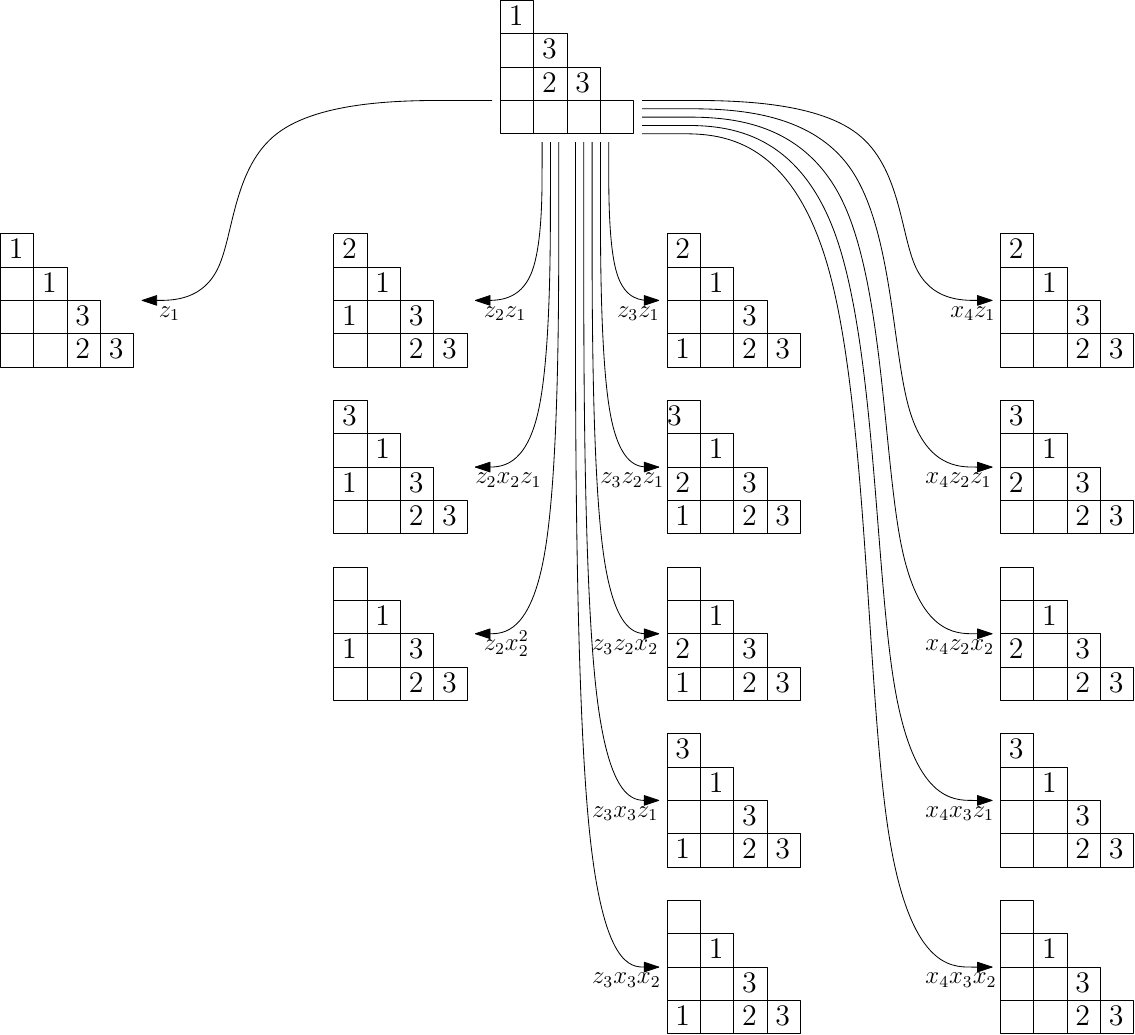}
	\caption{A state in $\mathcal{T}_4^4$ and all its successors. Here $x_i=1-y_{i-1}$.}
	\label{fig:tabtrans}
\end{figure}

The probability of such a transition $V\to W$ is then given by 
\begin{equation} \label{eq:tabltrans}
\mathcal{P}_{V,W} =
\begin{cases}
0 & \text{if $W_{i,j}\neq V_{i-1,j-1}$ for some $2\leq j\leq i \leq n$,} \\
\ds \prod_{i=1}^{T-1} \mathcal{C}_W(i,1) & \text{otherwise.}
\end{cases}
\end{equation}
This chain lumps to the overwriting model by the following procedure.
Let the rows of the tableaux be numbered from bottom to top.
For $V$ a tableau in $\mathcal{T}_n^T$, we define, for $k\in \{1,\ldots,n\}$,
$a_V(k)$ as the leftmost entry on the $(n-k+1)$'th row of $V$, and as $T$ otherwise.
The resulting lumped word $w \in St^T_n$ is then given by 
\begin{equation} \label{lumpstaircase}
w = a(V) := a_V(1)\cdots a_V(n).
\end{equation}
One can check that this procedure satisfies all the conditions for lumping; see Remark~\ref{rem:lumping}.
We are now in a position to prove Theorem~\ref{thm:special}.

\begin{proof}[Proof of Theorem~\ref{thm:special}]
The probability of having a $j$ in the last site of $w\in St_n^T$
in the overwriting chain is the same as the probability of having
a $j$ in the topmost cell of $V\in \mathcal{T}_n^T$ if $j<T$,
or of having nothing in this cell if $j=T$. This means
that any number $i<j$ failed to reach the first available cell, which happens
with probability $1-z_1$ for each one of them, and if $j<T$, $j$ reached that cell,
which happens with probability $z_1$, which proves the first part of the theorem.

For the second part of the theorem, we will only treat the case $i<j<T$,
since all the cases are proved in a similar fashion.
The joint probability of having a $i$ in the $(n-1)$'th site and
a $j$ in the $n$'th site of $w\in St_n^T$ is the same as the probability
of being in one of the two configurations of Figure \ref{fig:special}.

\begin{figure}[h]
	\centering
		\includegraphics{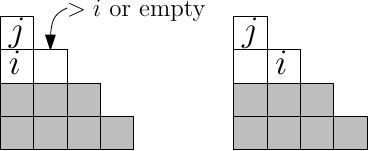}
	\caption{The two possible cases for $i<j<T$  in the proof of the second part of 
			Theorem~\ref{thm:special}.}
	\label{fig:special}
\end{figure}

The transition probability into the first configuration is
\begin{multline} \label{eq:conf1}
\underbrace{(1-z_1)^i}_{\substack{\text{no $k\leq i$ in the topmost} \\
\text{cell of the $2^{\text{nd}}$ column}}} \times
\underbrace{(1-y_2)^{i-1}}_{\substack{\text{no $k<i$ in the top $2$} \\
\text{cells of the $1^{\text{st}}$ column}}}
\times \underbrace{z_2}_{\substack{\text{$i$ is in the $2^{\text{nd}}$} \\
\text{cell of the $1^{\text{st}}$ column}}} \\
\times \underbrace{(1-z_1)^{j-i-1}}_{\substack{\text{no $k$ between $i+1$ and $j-1$ in the} \\ \text{topmost cell of the $1^{\text{st}}$ column}}}
\times \underbrace{z_1}_{\substack{\text{$j$ is in the topmost} \\
\text{cell of the $1^{\text{st}}$ column}}}
= (1-z_1)^{j-1}(1-y_2)^{i-1}z_1 z_2,
\end{multline}
while the transition probability into the second configuration is
\begin{multline} \label{eq:conf2}
\underbrace{(1-z_1)^{i-1}}_{\substack{\text{no $k<i$ in the topmost} \\
\text{cell of the $2^{\text{nd}}$ column}}} 
\times \underbrace{z_1}_{\substack{\text{$i$ is in the topmost} \\
\text{cell of the $2^{\text{nd}}$ column}}}
\times \underbrace{(1-y_2)^{i-1}}_{\substack{\text{no $k<i$ in the top $2$} \\
\text{cells of the $1^{\text{st}}$ column}}} \\
\times \underbrace{(1-z_1)^{j-i}}_{\substack{\text{no $k$ between $i$ and $j-1$ in the} \\ \text{topmost cell of the $1^{\text{st}}$ column}}}
\times \underbrace{z_1}_{\substack{\text{$j$ is in the topmost} \\
\text{cell of the $1^{\text{st}}$ column}}}
= (1-z_1)^{j-1}(1-y_2)^{i-1}z_1^2.
\end{multline}
By summing \eqref{eq:conf1} and \eqref{eq:conf2}, we get the desired probability.
\end{proof}

The idea of this proof also hints at why the stationary distribution of the overwriting Markov chain is not of a simple form. However, we will show that the stationary distribution of the enriched Markov chain on staircase tableaux has a particularly nice structure.

\begin{thm} \label{thm:tablstat}
The stationary distribution of $V\in \mathcal{T}_n^T$
for the staircase tableaux enriched  Markov chain is given by
\begin{equation}
\Pi(V) = \prod_{k=1}^n \prod_{i=1}^{T-1} \mathcal{C}_V(i,k),
\end{equation}
with the normalisation factor $(z_1+\cdots + z_{n+1})^{n(T-1)} = 1$.
\end{thm}

We will prove this formula by considering an even larger enlargement of the Markov chain on staircase tableaux, analogous to the doubly enriched chain of the single species annihilation model in \cite[Definition 4.14]{ABCN}. An immediate corollary of this result is the formula for the stationary distribution of the overwriting Markov chain.

\begin{cor} \label{cor:ovstat}
The stationary probability of $w=w_1\cdots w_n \in St_n^T$ for the overwriting
model is given by
\begin{equation} \label{eq:ovstat}
\pi(w)= \sum_{\substack{V \in \mathcal{T}_n^T \\ a(V) = w}} \Pi(V),
\end{equation}
where $a(V)$ is defined in \eqref{lumpstaircase}.
\end{cor}

\subsubsection{The doubly enriched chain}
\label{sec:trivia}

In this section we will construct a generalisation of the doubly enriched chain of 
the single species annihilation model in \cite[Definition 4.14]{ABCN}, which was a Markov chain on words of length $n$. In this case, the natural extension of this Markov chain is to matrices. The state space $\tilde{\mathcal{T}}_n^T$ is the set of matrices with $T-1$ rows and $n$ columns with entries in $\{1,\ldots,n+1\}$. 

It is clear there are $(n+1)^{n(T-1)}$ different states. As usual, given $z_1,\ldots, \allowbreak z_{n+1}$ nonnegative real numbers summing to $1$, the transitions are defined
as follows. For $M \in \tilde{\mathcal{T}}_n^T$, all transitions from 
$M$ are obtained by deleting the last column in $M$, shifting all the remaining columns to the right by one, and adding an arbitrary column on the left. The probability of this transition is given by the product of factors $z_i$ for each element $i$ in the resulting first column. More precisely,
\begin{equation} \label{eq:trivtrans}
\tilde{\mathcal{P}}_{M,N}=
\begin{cases}
\ds \prod_{i=1}^{T-1} z_{N_{i,1}} & \text{if $N_{k,l}= M_{k,l-1}$ for all $k\in \{1,\ldots,T-1\}$} \\
& \hfill \text{and all $l\in \{2,\ldots,n\}$,}\\
 0 & \text{otherwise.}
\end{cases}
\end{equation}
Since this is a product of $T-1$ independent copies of the single row Markov chain, it is clear this Markov chain is recurrent. Further, the stationary distribution is given, for $M\in \tilde{\mathcal{T}}_n^T$, by the product
\begin{equation} \label{eq:trivstat}
\tilde{\Pi}(M) = \prod_{i=1}^{T-1} \prod_{j=1}^n z_{M_{i,j}}
\end{equation}
with normalisation $(z_1 + \cdots + z_n)^{n(T-1)}=1$.\\

\begin{rem} \label{rem:ultrafast}
The dynamics of the doubly enriched chain guarantee that the stationary distribution is reached after $n$ steps (indeed, the first state has been completely forgotten after $n$ steps), which is the desired ultrafast convergence property. Equivalently, $n$ is a strong stationary time for this chain.
\end{rem}

An immediate consequence of Remark~\ref{rem:ultrafast} is a complete description of the spectrum of the transition matrix, given by the following theorem.

\begin{thm}
Let $n, T\in \mathbb{N}$, and let $M$ be the transition matrix of the doubly enriched chain on $\tilde{\mathcal{T}}_n^T$.
The eigenvalues for $M$ are $1$ with multiplicity $1$ and $0$ with
multiplicity $n-1$.
\end{thm}

\begin{proof}
As stated before, the stationary distribution is reached after $n$ steps.
This means that $M^n$ is the Matrix with all the rows being the left normalised eigenvector for $M$ (which represents the stationary distribution). Thus,
$M^{n+1}=M^n$, and therefore $X^{n+1}-X^n$ is a nullifying polynomial for $M$. This shows that $1$ is an eigenvalue of multiplicity $1$ (we already knew that its multiplicity was at least $1$) and that $0$ is the only other eigenvalue.
\end{proof}

The doubly enriched Markov chain on matrices described above lumps onto the singly-enriched chain on staircase tableaux. 
Let $M$ be in $\tilde{\mathcal{T}}_n^T$. 
We will construct a tableau $T$ associated to $M$ by starting with an empty tableau. We then fill $T$ according to the following pseudocode.

\noindent $\bullet$ for $k$ decreasing from $n$ down to $1$:\\
\indent $\bullet$ for $i$ increasing from $1$ to $T-1$:\\
\indent \indent  $\bullet$ if $M_{i,k}$ is less than or equal to the number of available positions in \\
\indent \indent \indent  the $k$'th column of $T$: \\
\indent \indent \indent \indent $\bullet$ insert $i$ in the $M_{i,k}$'th position from the top\\
\indent \indent $\bullet$ else: \\
\indent \indent \indent \indent $\bullet$ do not insert $i$. \\
We then set $A(M) := T$.
One can check that this algorithm leads to an actual lumping between these two Markov chains; see Remark~\ref{rem:lumping}. 

\begin{exmp}
Figure \ref{fig:owlump} gives the example of a matrix $M$ in $\tilde{\mathcal{T}}_4^3$ and
the tableau $T$ in $\mathcal{T}_4^3$ it lumps onto.
This tableau $T$ can then be lumped onto $w=1243$ in $St_4^3$.
\begin{figure}[ht]
	\centering
		\includegraphics{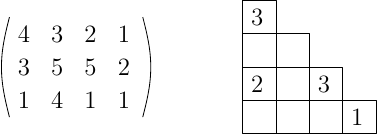}
	\caption{A matrix in $\tilde{\mathcal{T}}_4^3$ and the tableau in 
			$\mathcal{T}_4^3$ that it lumps to.}
	\label{fig:owlump}
\end{figure}
\end{exmp}

\begin{rem}
Since lumpings preserve the ultrafast convergence property, both the chain on staircase tableaux on $\mathcal{T}^T_n$ and the overwriting Markov chain on $St^T_n$
converge in $n$ steps.	
\end{rem}

We now prove the formula for the stationary distribution of the Markov chain on staircase tableaux.

\begin{proof}[Proof of Theorem~\ref{thm:tablstat}]
We have to check that, for each 	$V \in \mathcal{T}_n^T$,
\begin{equation}
\sum_{\substack{M \in \mathcal{T}^T_n \\ M \in A^{-1}(V)}} \tilde{\Pi}(M) = \Pi(V).
\end{equation}
To do so, let $k \in \{1,\ldots, n\}$ and $i \in \{1, \ldots, T-1\}$. If $i$ appears in the $k$'th column of $V$, then every $M$ in $\tilde{\mathcal{T}}_n^T$ projecting to $V$ must be such that $M_{i,k}$ is equal to the number of cells above entry $i$ in $V_{*,k}$ with no entry $j\leq i$ to the right.
Similarly if $i$ does not appear in the $k$'th column of $V$, then $M_{i,k}$ must be greater than the number of cells in $V_{*,k}$ with no entry $j\leq i$ in, to the right or atop of them.
In both cases, summing $z_\ell$ over all the possible values $\ell$ of $M_{i,k}$ gives us $C_V(i,k)$, and thus proves the result.
\end{proof}

\section{Several jugglers}
\label{sec:sevjug}

We now consider a completely different generalisation of Warrington's
model \cite{Warrington}. Instead of a multivariate or multispecies
generalisation, we will now consider that there are several jugglers,
and that each one of them can send the balls she catches to any other
juggler. We model this situation as follows. For $r,c,\ell$
nonnegative integers such that $\ell\leq rc$, we denote by $S_{r\times
  c}$ the set of rectangular arrays with $r$ rows and $c$ columns,
such that each cell either is empty or contains a ball, and by
$S_{r\times c,\ell} \subset S_{r \times c}$ the subset of arrays
containing exactly $\ell$ balls. Each column represents the balls that
are sent to a specific juggler. For $A$ and $B$ two arrays in
$S_{r\times c}$, we denote $A^-$ the array obtained by removing all
the balls in the lowest row, and moving all the other balls down one
row (hence the topmost row of $A^-$ is always empty). We write
$A\subset B$ if all the balls in $A$ are also in $B$.  For $i$ between
$1$ and $r$, we denote by $A_i$ the number of balls in the $i$'th row
(rows are numbered from top to bottom).

The several jugglers Markov chain is the Markov chain on the state space
$S_{r\times c,\ell}$ whose transition probabilities read, for 
$A,B \in S_{r\times c,\ell}$,
\begin{equation}
  \label{eq:sevjugp}
  \mathcal{P}_{A,B} = 
  \begin{cases}
    \ds \frac{1}{\binom{rc-\ell+A_r}{A_r}} & \text{if $A^-\subset B$},\\
    0 & \text{otherwise.}
  \end{cases}
\end{equation}
Here, $A_r$ is the number of balls in the lowest row of $A$, which is
exactly the number of balls the jugglers will have to send back. These
$A_r$ balls are reinjected uniformly in the $rc-\ell+A_r$ available
positions, under the constraint that no two balls go to the same
position. Note that there are no balls reinjected when $A_r = 0$, and
$\mathcal{P}_{A,A^-}=1$ in this case. The irreducibility and
aperiodicity of the several jugglers Markov chain are easy to check.

\begin{figure}
	\centering
		\includegraphics{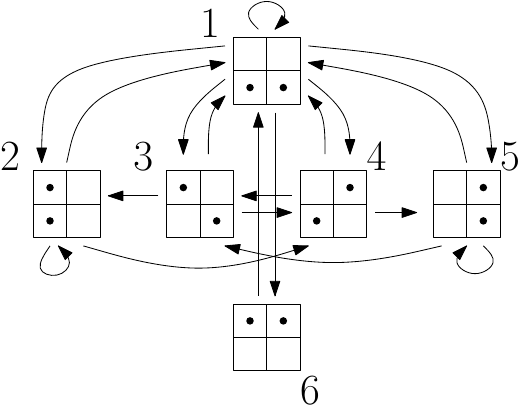}
	\caption{The several jugglers Markov chain on the state space $S_{2\times 2,2}$.}
	\label{fig:sevjug}
\end{figure}

\begin{exmp}
The transition Matrix of the several jugglers Markov chain on the state space
$S_{2\times 2,2}$ in the basis ordered $(1,2,3,4,5,6)$ on 
Figure \ref{fig:sevjug} reads
  \begin{equation}
    \begin{pmatrix}
     \frac{1}{6} & \frac{1}{6} & \frac{1}{6} & \frac{1}{6} & \frac{1}{6} & \frac{1}{6}\\
		 \frac{1}{3} & \frac{1}{3} & 0 & \frac{1}{3} & 0 & 0\\
		 \frac{1}{3} & \frac{1}{3} & 0 & \frac{1}{3} & 0 & 0\\
		 \frac{1}{3} & 0 & \frac{1}{3} & 0 & \frac{1}{3} & 0\\
		 \frac{1}{3} & 0 & \frac{1}{3} & 0 & \frac{1}{3} & 0\\
		 1 & 0 & 0 & 0 & 0 & 0\\
    \end{pmatrix}.
  \end{equation}
Note that $\left( 6,3,3,3,3,1 \right)$ is a left eigenvector for the eigenvalue $1$.
\end{exmp}

Again, we have an explicit expression for the stationary distribution
of this Markov chain. 

\begin{thm}
\label{thm:sevjugstat}

The stationary probability of $A \in S_{r\times c,\ell}$ for the
several jugglers Markov chain reads
\begin{equation}
  \label{eq:sevjugstat}
  \pi(A) = \frac{1}{Z_{r\times c,\ell}} \prod_{i=1}^r \left(ci-A_{<i}\right)_{A_i}
\end{equation}
where $A_{<i}=A_1+\cdots+A_{i-1}$ is the number of balls strictly above row $i$, $(x)_n=x(x-1)\cdots(x-n+1)$ is the Pochhammer symbol and $Z_{r\times c,\ell}$ is the normalisation factor.
\end{thm}

\begin{rem}
We have not been able to find a simple expression for the normalisation factor $Z_{r\times c,\ell}$.
\end{rem}

\begin{proof}
  We introduce an enriched chain as follows. A state in the enriched
  chain is an $(r+1)\times c$ array with $\ell$ arcs, each arc going
  between two cells that are not in the same row. For each arc, we mark
  the top cell with a cross and the bottom cell with a ball. Each cell
  can contain at most one cross and at most one ball, but could have
  one ball and one cross belonging to different arcs.  The projection
  to $S_{r\times c,\ell}$ is obtained by simply removing the top row,
  all arcs and crosses but leaving the balls in place. See Figure
  \ref{F:arcs} for all the states of the enriched chain projecting
  down to the rightmost state in Figure \ref{fig:sevjug}.

\begin{figure}[htpb]
\begin{center}
\begin{tikzpicture}  [scale=0.5, >=triangle 45]
\put (-120,0)
{\draw (0,0) grid (2,3);
\draw (1.5,0.5) node {$\bullet$};\draw (1.5,1.5) node {$\bullet$};
\draw (0.5,2.5) node {$\times$};\draw (1.5,2.5) node {$\times$};
\draw[-] (0.5,2.5) to [out = -90, in = 140] node{} (1.5,0.5);
\draw[-] (1.5,2.5) to [out = -90, in = 90] node{} (1.5,1.5);}
\put (-80,0)
{\draw (0,0) grid (2,3);
\draw (1.5,0.5) node {$\bullet$};\draw (1.5,1.5) node {$\bullet$};
\draw (0.5,2.5) node {$\times$};\draw (1.5,2.5) node {$\times$};
\draw[-] (1.5,2.5) to [out = -50, in = 50] node{} (1.5,0.5);
\draw[-] (0.5,2.5) to [out = -90, in = 180] node{} (1.5,1.5);}
\put (-40,0)
{\draw (0,0) grid (2,3);
\draw (1.5,0.5) node {$\bullet$};\draw (1.5,1.5) node {$\bullet$};
\draw (0.5,1.5) node {$\times$};\draw (1.5,2.5) node {$\times$};
\draw[-] (0.5,1.5) to [out = -90, in = 180] node{} (1.5,0.5);
\draw[-] (1.5,2.5) to [out = -90, in = 90] node{} (1.5,1.5);}
\put (0,0)
{\draw (0,0) grid (2,3);
\draw (1.5,0.5) node {$\bullet$};\draw (1.5,1.5) node {$\bullet$};
\draw (0.5,1.5) node {$\times$};\draw (0.5,2.5) node {$\times$};
\draw[-] (0.5,1.5) to [out = -90, in = 180] node{} (1.5,0.5);
\draw[-] (0.5,2.5) to [out = 0, in = 90] node{} (1.5,1.5);}
\put (40,0)
{\draw (0,0) grid (2,3);
\draw (1.5,0.5) node {$\bullet$};\draw (1.3,1.5) node {$\bullet$};
\draw (1.7,1.5) node {$\times$};\draw (1.5,2.5) node {$\times$};
\draw[-] (1.7,1.5) to [out = -90, in = 90] node{} (1.5,0.5);
\draw[-] (1.5,2.5) to [out = -90, in = 90] node{} (1.3,1.5);}
\put (80,0)
{\draw (0,0) grid (2,3);
\draw (1.5,0.5) node {$\bullet$};\draw (1.3,1.5) node {$\bullet$};
\draw (1.7,1.5) node {$\times$};\draw (0.5,2.5) node {$\times$};
\draw[-] (1.7,1.5) to [out = -90, in = 90] node{} (1.5,0.5);
\draw[-] (0.5,2.5) to [out = -90, in = 90] node{} (1.3,1.5);}
\end{tikzpicture}
\end{center}
\caption{All the states in the enriched chain projecting to the state in $S_{2\times 2,2}$ with both balls in the right column.}
\label{F:arcs}
\end{figure}
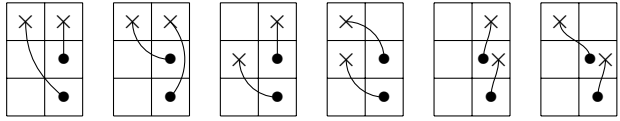

The transitions in the enriched chain are obtained by first moving all arcs, balls and crosses down one row in the array. Secondly, if there 
were balls in the bottom row, they and the corresponding arcs and crosses are removed. The balls are reinjected uniformly into the array 
except for the top row and under the condition that no two balls may be in the same cell, just like in the several jugglers Markov chain. 
For each of these balls an arc is inserted from the ball and up to a cross positioned uniformly in the top row under the condition that no two 
crosses can be in the same cell. (Alternatively we could define the transitions such that the new crosses in the top row appear in the same columns as the removed balls.)

Now we note that if we run the enriched chain backwards, it will be an identical chain with the roles of balls and crosses exchanged (turned upside down). The number of balls in the bottom row in a state is equal to the number of crosses in the top row for every state it may transition to. Also the number of ways to inject balls is the same as the number of ways of removing crosses. It follows that the number of transitions out of any state is equal to the number of transitions into the same state. Thus the stationary distribution is uniform for the enriched chain.

The uniformity of the enriched chain means that, to evaluate the
stationary probability $\pi(A)$ of a state of the several jugglers
Markov chain, it suffices to count the number of states in the
enriched chain projecting to it. For each ball we can place the cross
in any position in a row above with the constraint that no two crosses
can be in the same cell. The number of possibilities can be counted
row by row: assuming that the crosses corresponding to the balls
strictly above row $j$ have been chosen, there remains $cj-A_{<j}$
cells without crosses that may be matched with the $A_j$ balls in row
$j$, hence there are $(cj-A_{<j})_{A_j}$ possible choices for row $j$.
\end{proof}

\begin{rem}
J.S.~Kim \cite{JSK} has studied the model of a juggler with each site being allowed to contain up to a certain number $c>1$ of balls. Kim's model can be obtained from the several jugglers Markov chain by lumping.
\end{rem}

\section{Open Problems}

Several questions remain open in the multispecies juggling context.
We have not found an expression for the normalisation factor for the
juggling chain with several jugglers.  We have also not yet found a
multiparameter version for the latter model, as the possibility of
catching more than one ball at a time changes the behaviour quite
drastically. A multispecies model with several jugglers is one
possible extension of our model. From a probabilistic point of view,
it would also be natural to look at the extension to infinite models,
such as a Markov chain on the state space
$St_{n_1,\ldots,n_{T-1},\infty}$. This would contain as special
cases, the unbounded and infinite juggling models studied in \cite[Section 3]{ABCN}.

\bibliographystyle{plain}
\bibliography{multispecies}

\end{document}